\documentclass{amsart}
\usepackage{amsmath}
\usepackage{amssymb}
\usepackage{amsthm}
\usepackage{a4wide}
\usepackage{xcolor}
\usepackage{tikz}

\newtheorem{thm}{Theorem}
\newtheorem{lem}[thm]{Lemma}
\newtheorem{cor}[thm]{Corollary}
\newtheorem{prop}[thm]{Proposition}

\begin{document}


\title{Covering rectangles by few monotonous polyominoes}
\author{Christian Richter}
\address{Institute for Mathematics, Friedrich Schiller University, 07737 Jena, Germany}
\email{christian.richter@uni-jena.de}
\date{\today}

\begin{abstract}
A monotonous polyomino is formed by all lattice unit squares met by the graph of some fixed monotonous continuous function $f:[a,b] \to \mathbb{R}$ with $f(k) \notin \mathbb{Z}$ whenever $k \in \mathbb{Z}$. Our main result says that the least cardinality of a covering of a lattice $(m \times n)$-rectangle by monotonous polyominoes is $\left\lceil \frac{2}{3}\left(m+n-\sqrt{m^2+n^2-mn}\right)\right\rceil$.
The paper is motivated by a problem on arrangements of straight lines on chessboards.
\end{abstract}


\maketitle


\section{Introduction}

In \cite{barany_frankl} B\'ar\'any and Frankl pose the following question. 
\emph{What is the minimal number $p_n$ such that there exist $p_n$ straight lines that \emph{pierce} an $(n \times n)$-chessboard in the sense that each of its $n^2$ cells is met in its interior by at least one line?} Together with Ambrus and Varga they show in \cite{ambrus+} that $p_n > \frac{7}{10}n$ for sufficiently large $n$ and that $p_n \le n-1$ for $n \ge 3$, and they conjecture that $p_n=n-1$ for $n \ge 3$.

One can assume that none of the straight lines of the above piercing meets a vertex of a cell, since otherwise some sufficiently small shift of such a line would strictly enlarge the set of all cells met by it in the interior. Then each line gives rise to a polyomino on the chessboard, and the above problem amounts to a problem of optimal covers by polyominoes.
Recall that a \emph{polyomino} is a finite union of \emph{lattice unit squares} or \emph{lattice cells} (these are squares of the form $S{(k,l)}:=[k-1,k] \times [l-1,l]$, $k,l \in \mathbb{Z}$) with connected interior \cite[p.\ 3]{golomb}. Polyominoes give rise to many combinatorial problems, in particular to problems of tilings, see e.g.\ \cite{golomb,martin} and \cite[Section 9.4]{gruenbaum_shephard}.

The polyominoes in the above piercing problem satisfy a monotonicity condition depending on the slope of the underlying straight lines. We call a polyomino \emph{monotonically increasing} if it is composed of a finite sequence of lattice cells $S_1,\ldots,S_k$ such that $S_{j+1}$ is either the right-hand or the above neighbour of $S_j$, $j=1,\ldots,k-1$. In an analogous way we define \emph{monotonically decreasing polyominoes}.
Accordingly, a monotonous polyomino is formed by all lattice unit squares met by the graph of some fixed monotonous continuous function $f:[a,b] \to \mathbb{R}$ with $f(k) \notin \mathbb{Z}$ whenever $k \in \mathbb{Z}$. 
Monotonically increasing polyominoes appear in the literature under the names \emph{rim hooks}, \emph{skew hooks} or \emph{ribbon tiles} and have applications in algebraic combinatorics, see e.g.\ \cite{bertram+,james+,lascoux+,pak,sheffield,stanton_white}.

Here we study the following relative of the problem by B\'ar\'any and Frankl. \emph{What is the minimal number $\tilde{p}_n$ such that an $(n \times n)$-chessboard can be covered by $\tilde{p}_n$ monotonous polyominoes?} Clearly, $p_n \ge \tilde{p}_n$. We will see that $\tilde{p}_n=\left\lceil \frac{2}{3} n \right\rceil$ (see Corollary~\ref{cor:squares}), which is close to the above mentioned lower bound $p_n > \frac{7}{10}n$.

In fact, we will cover not only chessboards, but arbitrary \emph{lattice $(m \times n)$-rectangles}. These are integer translates of the standard rectangle $R_{m \times n}:=[0,m] \times [0,n]$ of height $n \in \{1,2,\ldots\}$ and length (or width) $m \in \{1,2,\ldots\}$. For formal reasons we define $R_{m \times n}:=\emptyset$ for $m,n \in \mathbb{N}:=\{0,1,2,\ldots\}$ with $\min\{m,n\}=0$.

Let us remark that the minimal number of monotonically increasing polyominoes for covering $R_{m \times n}$ is $\min\{m,n\}$: Indeed, $\min\{m,n\}$ stripes of size $\max\{m,n\} \times 1$ (or $1 \times \max\{m,n\}$, respectively) suffice to cover $R_{m \times n}$. A cover by less than $\min\{m,n\}$ tiles is impossible, since each monotonically increasing polyomino covers at most one of the cells $S(j,\min\{m,n\}+1-j) \subseteq R_{m \times n}$, $j=1,\ldots,\min\{m,n\}$.


\section{Main result}

Our main goal is the characterization of the numbers
\[
p(m,n):=\min\{p \in \mathbb{N}: \text{the rectangle } R_{m \times n}\text{ can be covered by } p \text{ monotonous polyominoes}\}
\]
for all $m,n \in \mathbb{N}$. 
Clearly, $p(m,n)=0$ if $\min\{m,n\}=0$. 
It turns out to be useful to start with the computation of the related quantities
\[
m(n,p):=\sup\{m \in \mathbb{N}: \text{the rectangle } R_{m \times n}\text{ can be covered by } p \text{ monotonous polyominoes}\}
\]
for $n,p \in \mathbb{N}$. Coverings of $R_{m \times n}$ by the $n$ horizontal polyominoes $[0,m] \times [j-1,j]$, $j=1,\ldots,n$, yield $m(n,p)=\infty$ for $p \ge n$. In the next section we shall prove the following.

\begin{thm} \label{thm:main}
{\bf (a) } For all $n,p \in \mathbb N$ with $n > p$,
$$
m(n,p)=p+\left\lfloor \frac{p^2}{4(n-p)} \right\rfloor.
$$

{\bf (b) } For all $m,n \in \mathbb N$, 
$$
p(m,n)= \left\lceil \frac{2}{3}\left(m+n-\sqrt{m^2+n^2-mn}\right)\right\rceil.
$$
\end{thm}

We stress two extremal cases. The first one concerns rectangles which are so flat that already their trivial covering by $n$ horizontal polyominoes is of minimal cardinality.

\begin{cor}
Let $m,n \in
\mathbb{N} \setminus \{0\}$ be such that $m \ge n$. Then the trivial covering of the rectangle $R_{m \times n}$ by $n$ parallel $(m \times 1)$-rectangles is of minimal cardinality among all coverings of $R_{m \times n}$ by monotonous polyominoes if and only if $m > \frac{(n+1)^2}{4}-1$ or, equivalently, if $n < 2\sqrt{m+1}-1$.
\end{cor}

\begin{proof}
We have to characterize the situation $p(m,n)=n$. Since always $p(m,n) \le n$, this is equivalent to $p(m,n) > n-1$, which is in turn equivalent to $\frac{2}{3}\left(m+n-\sqrt{m^2+n^2-mn}\right)> n-1$ by Theorem~\ref{thm:main}(b). This yields the claim.
\end{proof}

The second extremal case concerns squares, the least flat rectangles.

\begin{cor}\label{cor:squares}
$p(m,m)=\left\lceil \frac{2}{3}m\right\rceil$ for all $m \in \mathbb N$.
\end{cor}


\section{Verification of Theorem~\ref{thm:main}}


\subsection{Some notation and a refinement of Theorem~\ref{thm:main}(a)}

We see the rectangle $R_{m \times n}=[0,m] \times [0,n]$ as a union of $m$ \emph{columns} 
$$
C_k:=[k-1,k] \times [0,n]= S(k,1) \cup S(k,2) \cup \ldots \cup S(k,n),
$$
$k=1,\ldots,m$ (cf. Figure~\ref{fig:1}).
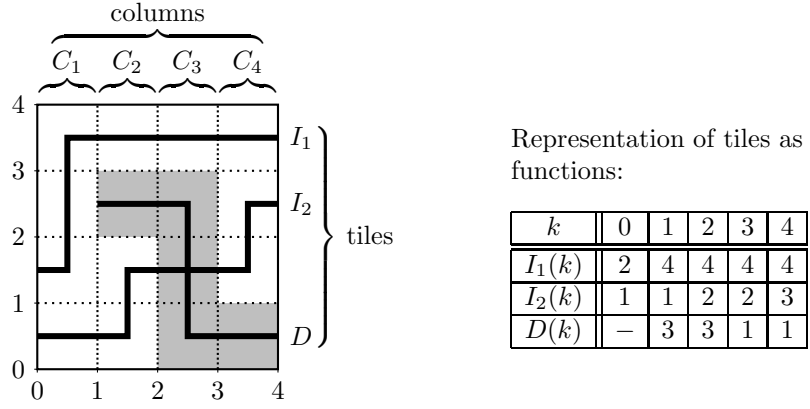
\begin{figure}
\begin{center}
\begin{tikzpicture}[xscale=.4,yscale=.44]

\fill[lightgray]
  (2,6)--(2,4)--(4,4)--(4,0)--(8,0)--(8,2)--(6,2)--(6,6)--cycle
  ;

\draw[thick]
  (0,0)--(8,0)--(8,8)--(0,8)--cycle 
	;

\draw[dotted,thick]
  (0,2)--(8,2)
  (0,4)--(8,4)
  (0,6)--(8,6)
	;

\draw[densely dotted, thick]
  (2,0)--(2,8)
	(4,0)--(4,8)
	(6,0)--(6,8)
	;

\draw[line width = .8mm]
  (0,1)--(3,1)--(3,3)--(7,3)--(7,5)--(8,5) node[right]{$I_2$}
	(0,3)--(1,3)--(1,7)--(8,7) node[right]{$I_1$}
	(2,5)--(5,5)--(5,1)--(8,1) node[right]{$D$}
  ;

\draw
  (0,0)--(0,-.1) node[below]{$0$}
  (2,0)--(2,-.1) node[below]{$1$}
  (4,0)--(4,-.1) node[below]{$2$}
  (6,0)--(6,-.1) node[below]{$3$}
  (8,0)--(8,-.1) node[below]{$4$}
  (0,0)--(-.1,0) node[left]{$0$}
  (0,2)--(-.1,2) node[left]{$1$}
  (0,4)--(-.1,4) node[left]{$2$}
  (0,6)--(-.1,6) node[left]{$3$}
  (0,8)--(-.1,8) node[left]{$4$}
	(1,8) node[above]{$\overbrace{\hspace*{7.5mm}}^{\textstyle C_1}$}
	(3,8) node[above]{$\overbrace{\hspace*{7.5mm}}^{\textstyle C_2}$}
	(5,8) node[above]{$\overbrace{\hspace*{7.5mm}}^{\textstyle C_3}$}
	(7,8) node[above]{$\overbrace{\hspace*{7.5mm}}^{\textstyle C_4}$}
	(4,9.5) node[above]{$\overbrace{\hspace*{28mm}}^{\textstyle\text{columns}}$}
	(7.9,4) node[right]{$\left.\begin{array}{c} \vspace{27mm} \end{array}\right\}$ tiles}
	;

\draw
  (15,4) node[right]{\begin{tabular}{l}
	  Representation of 
		tiles as\\ 
		functions:\\[2ex]
		$\begin{array}{|c||c|c|c|c|c|}
	  \hline
		k & 0 & 1 & 2 & 3 & 4 \\
		\hline\hline
		I_1(k) & 2 & 4 & 4 & 4 & 4 \\
		\hline
		I_2(k) & 1 & 1 & 2 & 2 & 3 \\
		\hline
		D(k) & - & 3 & 3 & 1 & 1 \\
		\hline
		\end{array}$
		\end{tabular}}
	;

\end{tikzpicture}
\end{center}
\caption{A covering of $R_{4 \times 4}$ with related notations and illustrations \label{fig:1}}
\end{figure}
A monotonous polyomino $P \subseteq R_{m \times n}$ (also called a \emph{tile}) that ranges from column $C_r$ to column $C_s$, $s \ge r$, is represented as a monotonous function 
$$
P:\{r-1,r,r+1,\ldots,s\}\to\{1,2,\ldots,n\}
$$ 
such that
$$
P \cap C_t=
\bigcup_{j=\min\{P(t-1),P(t)\}}^{\max\{P(t-1),P(t)\}}
S(t,j)=
\left\{\begin{array}{ll}
\bigcup_{j=P(t-1)}^{P(t)} S(t,j)
& \text{ if $P$ increases,}\\[1.5ex]
\bigcup_{j=P(t)}^{P(t-1)} S(t,j)
& \text{ if $P$ decreases}
\end{array}\right.
$$
for $t=r,r+1,\ldots,s$. Note that we use the same notation for a tile when it is considered as a union of lattice cells as well as when it is seen as a function. Visual illustrations of tiles can be given by rectilinear arcs passing through the centres of the respective lattice cells. For example, Figure~\ref{fig:1} shows a covering of $R_{4 \times 4}$ by two increasing tiles $I_1,I_2$ with the same domain $\{0,\ldots,4\}$ and a decreasing tile $D$ (greyed) with domain $\{1,\ldots,4\}$.

A covering of $R_{m \times n}$ by monotonous polyominoes is called an \emph{$(i,d)$-covering} if it consists of $i$ increasing and $d$ decreasing polyominoes. If some polyomino is both decreasing and increasing, we assign it (arbitrarily) to only one of the classes of increasing or decreasing polyominoes. In particular, an $(i,d)$-covering consists of $i+d$ tiles. 
For $n,i,d \in \mathbb{N}$ with $n \ge i+d$, we define a relative of $m(n,p)$ as
\[
m(n,i,d):=\sup\{m \in \mathbb{N}: \text{the rectangle } R_{m \times n}\text{ has an $(i,d)$-covering}\}.
\]
We shall show the following refinement of Theorem~\ref{thm:main}(a) before we come to the proof of Theorem~\ref{thm:main}.

\begin{prop}\label{prop}
For all $n,i,d \in \mathbb{N}$ with $n \ge i+d$,
$$
m(n,i,d)=i+d+\left\lfloor \frac{id}{n-(i+d)} \right\rfloor
$$
with the natural agreement that $m(n,i,d)=\infty$ if $i+d=n$; i.e., if the denominator on the right-hand side is $0$.
\end{prop}


\subsection{Proof of Proposition~\ref{prop}: upper estimate}

The main preparation of the proof is done by showing that, for any $i$ increasing polyominoes in $R_{m \times n}$, there exist $i' \le i$ increasing polyominoes in $R_{m \times n}$ that cover all the $i$ original polyominoes, do not overlap and are all defined on the full domain $\{0,\ldots,m\}$; i.e., contain cells of all the columns $C_1,\ldots,C_m$. For that, we first show that the original polyominoes can be replaced by ones with domain $\{0,\ldots,m\}$ and whose highest cells in column $C_m$ are mutually different (see Lemma~\ref{lem:upper_row} and Corollary~\ref{cor:7}). Then we show in Lemma~\ref{lem:8} that two increasing polyominoes with different final cell in $C_m$ can be replaced by two non-overlapping ones, before all polyominoes are made non-overlapping in Corollary~\ref{cor:9}.

\begin{lem}\label{lem:upper_row}
Let $I_1,\ldots,I_i \subseteq R_{m \times n}$ be increasing polyominoes, $m,n,i \in \mathbb{N} \setminus \{0\}$. Then there exist mutually distinct increasing polyominoes $I_0^1,I_1^1,\ldots,I^1_{i_1} \subseteq R_{m \times n}$, $0 \le i_1 < i$, that are fully defined on $\{0,\ldots,m\}$ and satisfy
\begin{equation}\label{eq:l6-1}
I_1 \cup \ldots \cup I_i \subseteq I_0^1 \cup \ldots\cup I^1_{i_1}
\end{equation}
as well as 
\begin{equation}\label{eq:l6-2}
I_1^1,\ldots,I_{i_1}^1 \subseteq R_{m \times (I_0^1(m)-1)}; \quad\text{that is,}\quad
\max\{I_1^1(m),\ldots,I_{i_1}^1(m)\} < I^1_0(m).
\end{equation}
\end{lem}

\begin{proof}
We suppose that all of $I_1,\ldots,I_i$ are fully defined on $\{0,\ldots,m\}$, since smaller tiles can be extended constantly to the left and to the right. W.l.o.g., $I_i$ is among those tiles from $I_1,\ldots,I_i$ that cover the largest portion of the upper part $[0,m] \times [\max\{I_1(m),\ldots,I_i(m)\}-1,n]$ of $R_{m \times n}$. Then $I_i(m)=\max\{I_1(m),\ldots,I_i(m)\}$ and, by monotonicity of the tiles,
\begin{equation}\label{eq:l6-3}
(I_1 \cup \ldots \cup I_i) \cap ([0,m] \times [\max\{I_1(m),\ldots,I_i(m)\}-1,n]) \subseteq I_i.
\end{equation}

We put $I_0^1:=I_i$ and 
$$
I^1_k(\cdot):=\min\{I_k(\cdot),\max\{I_1(m),\ldots,I_i(m)\}-1\}
$$ 
for $k=1,\ldots,i-1$, provided $\max\{I_1(m),\ldots,I_i(m)\} \ge 2$. If $\max\{I_1(m),\ldots,I_i(m)\} = 1$ we put $i_1:=0$. If some of $I^1_1,\ldots,I^1_{i-1}$ coincide, we keep only one of them; that is, w.l.o.g., $I^1_1,\ldots,I^1_{i_1}$ represent all of $I^1_1,\ldots,I^1_{i-1}$, where $0 \le i_1 \le i-1$.

The last definition yields directly \eqref{eq:l6-2}. Moreover, \eqref{eq:l6-1} is obtained by
\[
\begin{array}{rcl}
I_1 \cup \ldots \cup I_i 
&=&\big((I_1 \cup \ldots \cup I_i) \cap ([0,m] \times [\max\{I_1(m),\ldots,I_i(m)\}-1,n])\big)\\
&& \cup \left((I_1 \cup \ldots \cup I_i) \cap R_{m \times (\max\{I_1(m),\ldots,I_i(m)\}-1)}\right)\\
&\stackrel{\eqref{eq:l6-3}}{\subseteq}& I^1_0 \cup \left((I_1^1 \cup \ldots \cup I_{i-1}^1) \cup I_0^1\right)\\
&=& I^1_0 \cup \ldots \cup I^1_{i_1}.
\end{array}
\]
\end{proof}

\begin{cor}\label{cor:7}
Let $I_1,\ldots,I_i \subseteq R_{m \times n}$ be increasing polyominoes, $m,n,i \in \mathbb{N} \setminus \{0\}$. Then there exist increasing polyominoes $I'_1,\ldots,I'_{i'} \subseteq R_{m \times n}$, $i' \le i$, that 
 are fully defined on $\{0,\ldots,m\}$ and satisfy
\begin{equation}\label{eq:cor7-1}
I_1 \cup \ldots \cup I_i \subseteq I'_1 \cup \ldots \cup I'_{i'}
\end{equation}
as well as
\begin{equation}\label{eq:cor7-2}
I_1'(m) > I_2'(m) > \ldots > I'_{i'}(m).
\end{equation}
\end{cor}

\begin{proof}
We apply Lemma~\ref{lem:upper_row} to the tiles $I_1,\ldots,I_i$ on $R_{m \times n}$ and obtain $I_0^1 \subseteq R_{m \times n}$ and $I_1^1,\ldots,I_{i_1}^1 \subseteq R_{m \times n_1}$ with $n_1=\max\{I_1(m),\ldots,I_i(m)\}-1=I_0^1(m)-1$ and $i_1 < i$. Next we apply Lemma~\ref{lem:upper_row} to the tiles $I^1_1,\ldots,I^1_{i_1}$ on $R_{m \times n_1}$ and obtain $I_0^2 \subseteq R_{m \times n_1}$ and $I_1^2,\ldots,I_{i_2}^2 \subseteq R_{m \times n_2}$ with $n_2=\max\{I^1_1(m),\ldots,I^1_{i_1}(m)\}-1=I_0^2(m)-1$ and $i_2 < i_1$.
We apply this procedure $j$ times until $i_{j}=0$. Then we put $i':=j$ and $I_l':=I_0^l$, $l=1,\ldots,i'$.

Clearly, $i'=j \le i$, because $i>i_1>i_2>\ldots>i_j=0$. Inequalities \eqref{eq:cor7-2} follow from 
$$
I'_{l+1}(m)=I_0^{l+1}(m) \le n_l=I_0^l(m)-1=I'_l(m)-1<I'_l(m)
$$
for $l=1,\ldots,j-1$. Finally, \eqref{eq:cor7-1} is obtained by
$$
I_1 \cup \ldots \cup I_i \subseteq I_0^1 \cup (I_1^1\cup\ldots\cup I_{i_1}^1) \subseteq I_0^1 \cup I_0^2 \cup (I_1^2\cup\ldots\cup I_{i_2}^2) \subseteq \ldots \subseteq I_0^1 \cup I_0^2 \cup \ldots \cup I_0^{j}.
$$
\end{proof}

\begin{lem}\label{lem:8}
Let $I,J:\{0,\ldots,m\}\to\{1,\ldots,n\}$ be increasing polyominoes in $R_{m \times n}$ with $I(m) \ne J(m)$. Then $L=L[I,J]$ and $U=U[I,J]$, defined for $k \in \{0,\ldots,m\}$ by
\begin{align*}
L(k)&:=\min\{I(k),J(k)\},\\
U(k)&:=\left\{
\begin{array}{ll}
\max\big\{I(k),J(k),\min\{I(k+1),J(k+1)\}+1\big\},& k < m,\\
\max\{I(k),J(k)\}, & k=m,
\end{array}
\right.
\end{align*}
satisfy the following:
\begin{itemize}
\item[(i)]
$L$ and $U$ are increasing polyominoes,
\item[(ii)]
$I \cup J \subseteq L \cup U \subseteq R_{m \times n}$,
\item[(iii)]
$L$ and $U$ do not overlap and $L$ is below $U$; that is,
$L(k) < U(k-1)$ for $k=1,\ldots,m$.
\end{itemize}
\end{lem}

\begin{proof}
\emph{Claim (i).}
Since $I$ and $J$ are increasing, it is clear that $L$ is increasing and that $U(k-1) \le U(k)$ for $k=1,\ldots,m-1$. Finally, $I(m) \ne J(m)$ implies
\begin{align*}
U(m-1)&=\max\big\{I(m-1),J(m-1),\min\{I(m),J(m)\}+1\big\}\\
&\le\max\big\{I(m-1),J(m-1),\max\{I(m),J(m)\}\big\}\\
&=U(m).
\end{align*}

\emph{Claim (ii).} By $I,J \subseteq R_{m \times n}$ and by claim (i), $1 \le L(k) \le n$ and $1 \le U(k) \le U(m) \le n$. Thus, $L \cup U \subseteq R_{m \times n}$.

To see the inclusion $I \cup J \subseteq L \cup U$, we consider the columns $C_k$, $k=1,\ldots,m$, of $R_{m \times n}$ separately. We need to show that $(I \cup J) \cap C_k \subseteq (L \cup U) \cap C_k$.  
Recall that $I \cap C_k$ consists of the cells $S(k,I(k-1)),S(k,I(k-1)+1),\ldots,S(k,I(k))$; and analogously for $J,L,U$.

\emph{Case 1: $I(k) < J(k-1)$. } Then $I(k) < J(k-1) \le J(k)$, so that
\begin{equation}\label{eq:lem8-1}
\min\{I(k),J(k)\}=I(k)\le J(k-1)-1.
\end{equation}

Now $I \cap C_k \subseteq L \cap C_k$, because $L(k-1) \le I(k-1)$ and $L(k) \stackrel{\eqref{eq:lem8-1}}{=}I(k)$. Also $J \cap C_k \subseteq U \cap C_k$, since $U(k) \ge J(k)$ and 
$$
U(k-1) \stackrel{\eqref{eq:lem8-1}}{\le}\max\{I(k-1),J(k-1),J(k-1)\}\le\max\{I(k),J(k-1)\}\stackrel{(\text{Case 1})}{=}J(k-1).
$$
Accordingly, $(I \cup J) \cap C_k \subseteq (L \cup U) \cap C_k$.

\emph{Case 2: $J(k) < I(k-1)$. } This is analogous to Case 1.

\emph{Case 3: $I(k) \ge J(k-1)$ and $J(k) \ge I(k-1)$. } By monotonicity,
\begin{equation}\label{eq:lem8-2}
\min\{I(k),J(k)\} \ge \max\{I(k-1),J(k-1)\}.
\end{equation}

By the three observations $L(k-1)=\min\{I(k-1),J(k-1)\}$,
\begin{eqnarray*}
U(k-1)&=&\max\big\{\max\{I(k-1),J(k-1)\},\min\{I(k),J(k)\}+1\big\}\\
&\stackrel{\eqref{eq:lem8-2}}{\le}&
\max\big\{\min\{I(k),J(k)\},\min\{I(k),J(k)\}+1\big\}\\
&=&L(k)+1,
\end{eqnarray*}
and $U(k) \ge \max\{I(k),J(k)\}$, we get
$$
(L \cup U)\cap C_k \supseteq S(k,\min\{I(k-1),J(k-1)\}) 
\cup  \ldots \cup S(k,\max\{I(k),J(k)\}) \supseteq (I \cup J) \cap C_k.
$$

\emph{Claim (iii).} For $k=1,\ldots,m$, 
$$
U(k-1) \ge \min\{I(k),J(k)\}+1= L(k)+1 > L(k).
$$
That is, the lowest cell of $U$ in column $C_k$ is placed strictly higher than the highest cell of $L$ in $C_k$. This gives (iii).
\end{proof}

\begin{cor}\label{cor:9}
Let $I_1,\ldots,I_i \subseteq R_{m \times n}$ be increasing polyominoes, $m,n,i \in \mathbb{N} \setminus \{0\}$. Then there exist non-overlapping increasing polyominoes $I'_1,\ldots,I'_{i'} \subseteq R_{m \times n}$, $0 \le i' \le i$, that are fully defined on $\{0,\ldots,m\}$ and satisfy
$$
I_1 \cup \ldots \cup I_i \subseteq I'_1 \cup \ldots \cup I'_{i'}.
$$
\end{cor}

\begin{proof}
By Corollary~\ref{cor:7}, there are increasing polyominoes $I_1^0,\ldots,I_{i'}^0 \subseteq R_{m \times n}$, $i'\le i$, that are fully defined on $\{0,\ldots,m\}$ such that 
$$
I_1 \cup \ldots \cup I_i \subseteq I_1^0 \cup\ldots\cup I_{i'}^0 
\quad\text{and}\quad I_1^0(m) < \ldots < I_{i'}^0(m).
$$ 

In a first step we use Lemma~\ref{lem:8} for defining
$$
\begin{array}{ll}
L_2:=L[I^0_1,I^0_2],& I^1_2:=U[I^0_1,I^0_2],\\
L_3:=L[L_2,I^0_3],& I^1_3:=U[L_2,I^0_3],\\
L_4:=L[L_3,I^0_4],& I^1_4:=U[L_3,I^0_4],\\ 
\ldots,\\
L_{i'}:=L[L_{i'-1},I^0_{i'}],& I^1_{i'}:=U[L_{i'-1},I^0_{i'}]\\
\text{and} \quad 
I'_1:=L_{i'}.
\end{array}
$$
Then 
$$
I^0_1 \cup \ldots \cup I^0_{i'} \subseteq I_1' \cup I^1_2 \cup\ldots\cup I^1_{i'}, \quad I'_1(m)=I^0_1(m) < I^1_2(m)=I^0_2(m)< \ldots < I^1_{i'}(m)=I^0_{i'}(m)
$$
and $I'_1=\min\{I^0_1,\ldots,I^0_{i'}\}$ is without overlap below each of $I^1_2,\ldots,I^1_{i'}$.

The second step acts on $I^1_2,\ldots,I^1_{i'}$ as the first one did on $I^0_1,\ldots,I^0_{i'}$. It gives $I'_2$ and $I^2_3,\ldots,I^2_{i'}$ such that 
$$
I^1_2\cup\ldots\cup I^1_{i'} \subseteq I_2' \cup I^2_3 \cup\ldots\cup I^2_{i'}, \quad I'_2(m)=I^1_2(m) < I^2_3(m)=I^1_3(m)< \ldots < I^2_{i'}(m)=I^1_{i'}(m)
$$
and $I'_2=\min\{I^1_2,\ldots,I^1_{i'}\}$ is without overlap below each of $I^2_3,\ldots,I^2_{i'}$ as well as above $I'_1$.

Execution of $i'$ such steps produces the required non-overlapping polyominoes $I'_1,\ldots,I'_{i'}$.
\end{proof}

\begin{lem}\label{lem:10}
Let $n,i,d \in \mathbb{N}$ be such that $i+d < n$. Then $m(n,i,d) < \infty$.
\end{lem}

\begin{proof}
Consider an $(i,d)$-covering of $R_{m \times n}$. Since each tile  covers at most $m+n-1$ cells of $R_{m \times n}$ and since the $i+d$ tiles cover $R_{m \times n}$ completely, we have
$mn \le (i+d)(m+n-1)$. This yields $m \le \frac{(i+d)(n-1)}{n-(i+d)}$. Thus, $m(n,i,d)\le \frac{(i+d)(n-1)}{n-(i+d)}<\infty$. 
\end{proof}

\begin{lem}\label{lem:11}
Let $n,i,d \in \mathbb{N}$ be such that $i+d < n$ and let $m=m(n,i,d)$. If $I$ is an increasing and $D$ is a decreasing polyomino of an $(i,d)$-covering of $R_{m \times n}$ such that both $I$ and $D$ are fully defined on $\{0,\ldots,m\}$, then 
$$
I(0) < D(0) \quad\text{and}\quad I(m) > D(m).
$$
\end{lem}

\begin{proof}
Assume that $I(0) \ge D(0)$. Then $I$ and $D$ can be extended to $\{-1,0,\ldots,m\}$ by putting $I(-1):=1$ and $D(-1):=n$. These extensions cover the column $C_{0}:=[-1,0] \times [0,n]$. Thus there exists an $(i,d)$-covering of the $((m+1) \times n)$-rectangle $C_0 \cup R_{m \times n}$, contradicting $m=m(n,i,d)$. 

The second inequality is shown analogously.
\end{proof}

\begin{proof}[Proof of Proposition~\ref{prop}, upper estimate]
We can suppose $i+d < n$, since otherwise the upper estimate is $\infty$. Let $m=m(n,i,d)$; cf.\ Lemma~\ref{lem:10}. By Corollary~\ref{cor:9}, there exists an $(i',d')$-covering of $R_{m \times n}$ by $0 \le i' \le i$ increasing tiles $I_1, \ldots,I_{i'}:\{0,\ldots,m\}\to\{1,\ldots,n\}$ that are mutually non-overlapping and $0 \le d' \le d$ decreasing tiles $D_1,\ldots,D_{d'}:\{0,\ldots,m\}\to\{1,\ldots,n\}$ that are mutually non-overlapping. 

We denote the number of lattice cells that constitute a tile $I$ by $\# I$. Using the monotonicity of the tiles, we get
$$
\# I_j= m+I_j(m)-I_j(0),\; 1 \le j \le i'.
$$
Since $I_1,\ldots,I_{i'}$ do not overlap, we have, w.l.o.g.,
$$
1 \le I_1(0) < I_2(0) < \ldots < I_{i'}(0) \quad\text{and}\quad n \ge I_{i'}(m) > I_{i'-1}(m) > \ldots > I_{1}(m).
$$
Then the total number of cells given by the increasing tiles is
\begin{align*}
\sum_{j=1}^{i'} \# I_j &= i'm+\big(I_{i'}(m)+I_{i'-1}(m)+\ldots+I_1(m)\big)-\big(I_1(0)+I_2(0)+\ldots+I_{i'}(0)\big)\\ 
&\le i'm+\big(n+(n-1)+\ldots+(n-i'+1)\big)-(1+2+\ldots+i')\\
&=i'm+\big((n-i')+((n-1)-(i'-1))+\ldots+((n-i'+1)-1)\big)\\
&=i'm+i'(n-i')\\
&=i'(m+n-i').
\end{align*}
Similarly, the total number of cells of the decreasing tiles is
$$
\sum_{k=1}^{d'} \# D_k \le d'(m+n-d').
$$

If one of the $mn$ cells from $R_{m \times n}$ is covered by more than one of the given tiles, then it is covered by exactly one increasing and exactly one decreasing tile, since neither two increasing tiles nor two decreasing tiles overlap. By Lemma~\ref{lem:11}, every increasing tile crosses every decreasing tile in at least one cell. There are at least $i
'd'$ such crossings. Thus, the total number of squares of tiles, counted with multiplicities, satisfies
$$
mn+i'd' \le \sum_{j=1}^{i'} \# I_j + \sum_{k=1}^{d'} \# D_k \le 
i'(m+n-i')+d'(m+n-d').
$$
This implies
$$
m \le i'+d'+\frac{i'd'}{n-(i'+d')} \le i+d+\frac{id}{n-(i+d)},
$$ 
and the upper estimate of Proposition~\ref{prop} is verified.
\end{proof}


\subsection{Proof of Proposition~\ref{prop}: lower estimate}

The proof of the lower estimate rests mainly on an inductive construction based on two estimates from Lemma~\ref{lem:tilde_m}, who are obtained by extending coverings from smaller rectangles to larger ones. These extensions are possible, because the tiles of an $(i,d)$-covering of $R_{m \times n}$ can be assumed to have a particular beginning in the left-most column $C_1$, as is shown in the following.

\begin{lem}\label{lem:12}
Let $m,n,i,d \in \mathbb{N}$ be such that $m,n \ge 1$ and $n\ge i+d$. If $R_{m \times n}$ has an $(i,d)$-covering, then $R_{m \times n}$ has an $(i,d)$-covering consisting of increasing tiles $I_1,\ldots,I_i: \{0,\ldots,m\}\to\{1,\ldots,n\}$ and decreasing tiles $D_1,\ldots,D_d: \{0,\ldots,m\}\to\{1,\ldots,n\}$ such that
$$
I_j(0)=j, \; 1 \le j \le i, \quad\text{and}\quad D_k(0)=n+1-k,\; 1 \le k \le d.
$$
\end{lem}

\begin{proof}
By Corollary~\ref{cor:9}, $R_{m \times n}$ has a covering by $i'\le i$ non-overlapping increasing tiles $I'_1,\ldots,I'_{i'}$ and $d' \le d$ non-overlapping decreasing tiles $D'_1,\ldots,D'_{d'}$. W.l.o.g.,
$1 \le I'_1(0) < \ldots < I'_{i'}(0)$ and $n \ge D'_1(0) > \ldots > D'_{d'}(0)$, whence
$$
I'_j(0) \ge j, \; 1 \le j \le i', \quad\text{and}\quad D'_k(0) \le n+1-k, \; 1 \le k \le d'.
$$
We define $I_j(l):=I'_j(l)$, $1 \le l \le m$, and $I_j(0):=j$, this way defining an increasing tile $I_j \supseteq I'_j$. Similarly, we put $D_k(l):=D'_k(l)$, $1 \le l \le m$, and $D_k(0):=n+1-k$, this way defining a decreasing tile $D_k \supseteq D'_k$. So we find a covering of $R_{m \times n}$ by $I_1,\ldots,I_{i'}$ and $D_1,\ldots,D_{d'}$ such that
$$
I_j(0)=j, \; 1 \le j \le i', \quad\text{and}\quad D_k(0)=n+1-k,\; 1 \le k \le d'.
$$
If $i'<i$ we add the constant (to be considered as increasing) tiles $I_j\equiv j$, $i' < j \le i$. If $d'<d$ we add the constant (to be considered as decreasing) tiles $D_k\equiv n+1-k$, $d' < k \le d$.
\end{proof}

In the following lemma we use the quantity
$$
\tilde{m}(e,i,d):=m(e+i+d,i,d).
$$
That is, $\tilde{m}(e,i,d)$ is the maximal length $m$ of a rectangle $R_{m \times n}$ that can be covered by $i$ increasing and $d$ decreasing polyominoes and whose height $n=e+i+d$ exceeds $i+d$ by $e$.

\begin{lem}\label{lem:tilde_m}
For all $e,i,d \in \mathbb{N}$,
\begin{equation}\label{eq:tilde_m>}
\tilde{m}(e,i,d) \ge \left\{
\begin{array}{ll}
d+e+\tilde{m}(e,i-e,d)& \text{if } i \ge e,\\
i+e+\tilde{m}(e,i,d-e)& \text{if } d \ge e.
\end{array}
\right.
\end{equation}
\end{lem}

\begin{proof}
The claim is obvious for $e=0$, because $\tilde{m}(0,i,d)=m(i+d,i,d)=\infty$. 

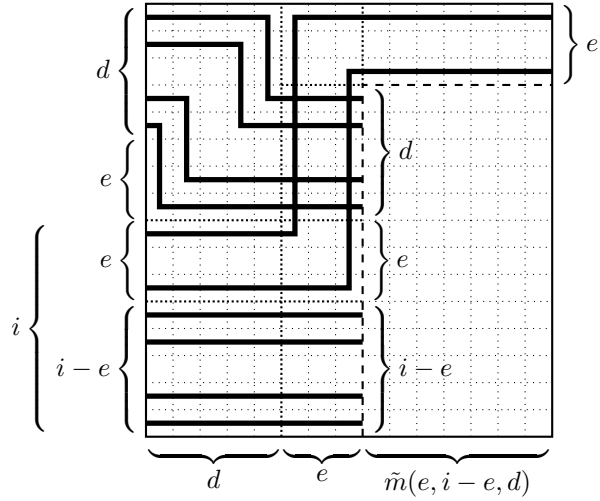
\begin{figure}
\begin{center}
\begin{tikzpicture}[xscale=.36,yscale=.36]

\draw[thick]
	(0,0)--(0,16)--(15,16)--(15,0)--cycle
	(-.5,13.5) node[left] {$d\left\{\begin{array}{c} \vspace{35pt} \end{array}\right.\hspace{-17pt}$}
	(-.5,9.5) node[left] {$e\left\{\begin{array}{c} \vspace{17pt} \end{array}\right.\hspace{-17pt}$}
	(-.5,6.5) node[left] {$e\left\{\begin{array}{c} \vspace{17pt} \end{array}\right.\hspace{-17pt}$}
	(-.5,2.5) node[left] {$i-e\left\{\begin{array}{c} \vspace{35pt} \end{array}\right.\hspace{-17pt}$}
	(-3.8,4) node[left] {$i\left\{\begin{array}{c} \vspace{68pt} \end{array}\right.\hspace{-17pt}$}
	(8.5,10.5) node[right] {$\hspace{-17pt}\left.\begin{array}{c} \vspace{35pt} \end{array}\right\}d$}
	(8.5,6.5) node[right] {$\hspace{-17pt}\left.\begin{array}{c} \vspace{17pt} \end{array}\right\}e$}
	(8.5,2.5) node[right] {$\hspace{-17pt}\left.\begin{array}{c} \vspace{35pt} \end{array}\right\}i-e$}
	(15.5,14.5) node[right] {$\hspace{-17pt}\left.\begin{array}{c} \vspace{17pt} \end{array}\right\}e$}
	(2.5,0) node[below] {$\underbrace{\hspace{49pt}}_{\textstyle d}$}
	(6.5,0) node[below] {$\underbrace{\hspace{29pt}}_{\textstyle e}$}
	(11.5,0) node[below] {$\underbrace{\hspace{69pt}}_{\textstyle\tilde{m}(e,i-e,d)}$}
	;

\draw[dashed,thick]
  (8,0)--(8,13)--(15,13)
  ;

\draw[dotted]
  (0,1)--(15,1)
  (0,2)--(15,2)
  (0,3)--(15,3)
  (0,4)--(15,4)
  (8,5)--(15,5)
  (0,6)--(15,6)
  (0,7)--(15,7)
  (8,8)--(15,8)
  (0,9)--(15,9)
  (0,10)--(15,10)
  (0,11)--(15,11)
  (0,12)--(15,12)
  (0,13)--(5,13)
  (0,14)--(15,14)
  (0,15)--(15,15)
  (1,0)--(1,16)
  (2,0)--(2,16)
  (3,0)--(3,16)
  (4,0)--(4,16)
  (6,0)--(6,16)
  (7,0)--(7,16)
  (9,0)--(9,16)
  (10,0)--(10,16)
  (11,0)--(11,16)
  (12,0)--(12,16)
  (13,0)--(13,16)
  (14,0)--(14,16)
  ;

\draw[densely dotted, thick]
  (5,0)--(5,16)
	(8,13)--(8,16)
	(0,5)--(8,5)
	(0,8)--(8,8)
	(5,13)--(8,13)
	;
	
\draw[line width = .7mm] 
  (0,15.5)--(4.5,15.5)--(4.5,12.5)--(8,12.5)
	(0,14.5)--(3.5,14.5)--(3.5,11.5)--(8,11.5)
  (0,12.5)--(1.5,12.5)--(1.5,9.5)--(8,9.5)
	(0,11.5)--(.5,11.5)--(.5,8.5)--(8,8.5)
	(0,7.5)--(5.5,7.5)--(5.5,15.5)--(15,15.5)
	(0,5.5)--(7.5,5.5)--(7.5,13.5)--(15,13.5)
	(0,4.5)--(8,4.5)
	(0,3.5)--(8,3.5)
	(0,1.5)--(8,1.5)
	(0,.5)--(8,.5)
  ;

\end{tikzpicture}
\end{center}
\caption{The first estimate in Lemma~\ref{lem:tilde_m}\label{fig:3}}
\end{figure}
For $e \ge 1$, it is shown in Figure~\ref{fig:3} that a $\big((d+e+\tilde{m}(e,i-e,d)) \times (e+i+d)\big)$-rectangle has an $(i,d)$-covering, provided that $i \ge e$: The first $d$ columns are used to decrease the vertical level of the $d$ decreasing tiles by $e$ units. In the next $e$ columns the upper $e$ increasing tiles are lifted by $d+e$ units. In the last $\tilde{m}(e,i-e,d)$ columns the last mentioned $e$ increasing tiles keep their vertical levels constant. The remaining $i-e$ increasing and $d$ decreasing tiles are arranged such that they can be continued by an $(i-e,d)$-covering of the dashed residual $\big(\tilde{m}(e,i-e,d) \times (i+d)\big)$-rectangle, which exists by the definition of $\tilde{m}(e,i-e,d)$ and by Lemma~\ref{lem:12}. 

So the upper inequality of \eqref{eq:tilde_m>} is verified. The lower one follows by the symmetry in $i$ and $d$.
\end{proof}

\begin{proof}[Proof of Proposition~\ref{prop}, lower estimate]
We have to show that
\begin{equation}\label{eq:goal}
\tilde{m}(e,i,d)\ge i+d+\left\lfloor \frac{id}{e} \right\rfloor
\end{equation}
for all $e,i,d \in \mathbb{N}$. We proceed by induction on $e$.

The base case with $e=0$ is trivial: $\tilde{m}(0,i,d)=m(i+d,i,d)=\infty$. In order to consider a less degenerate situation, we note that the claim $\tilde{m}(1,i,d)\ge i+d+id$ for $e=1$ is obtained by $i$-fold application of the first line of \eqref{eq:tilde_m>}, then $d$-fold application of the second line \eqref{eq:tilde_m>} and the final observation that $\tilde{m}(1,0,0)=0$.

We come to the induction step (with $e \ge 1$): Put $i=\alpha e+\iota$, $d=\beta e+\delta$ with integers $\alpha,\beta,\iota,\delta$ and $0 \le \iota,\delta < e$. 
First we show that 
\begin{equation}\label{eq:ide}
\tilde{m}\left(\left\lfloor\frac{\iota \delta}{e}\right\rfloor,\iota,\delta\right) \ge \iota+\delta+e.
\end{equation}
Indeed, this is trivial if $\iota\delta=0$, since then $\tilde{m}(0,\iota,\delta)=m(i+d,i,d)=\infty$. If $\iota\delta >0$, we have $\left\lfloor\frac{\iota \delta}{e}\right\rfloor \le \frac{\iota}{e}\delta < 1e=e$, and the induction hypothesis yields
$$
\tilde{m}\left(\left\lfloor\frac{\iota \delta}{e}\right\rfloor,\iota,\delta\right) \ge \iota+\delta+\left\lfloor\frac{\iota \delta}{\left\lfloor\frac{\iota \delta}{e}\right\rfloor}\right\rfloor \ge \iota+\delta+e.
$$ 
So \eqref{eq:ide} is shown.

Note that \eqref{eq:ide} says that a $\left((\iota+\delta+e) \times \left(\iota+\delta+\left\lfloor \frac{\iota\delta}{e} \right\rfloor\right)\right)$-rectangle has a $(\iota,\delta)$-covering. By reflecting w.r.t.\ the straight line $y=-x$ we get a $(\iota,\delta)$-covering of a $\left(\left(\iota+\delta+\left\lfloor \frac{\iota\delta}{e} \right\rfloor\right) \times (\iota+\delta+e)\right)$-rectangle. Consequently, 
\begin{equation}\label{eq:ide2}
\tilde{m}(e,\iota,\delta) \ge \iota+\delta+\left\lfloor \frac{\iota\delta}{e} \right\rfloor.
\end{equation}

Now we estimate
\begin{eqnarray*}
\tilde{m}(e,i,d)&=&\tilde{m}(e,\iota +\alpha e, d)\\
& \stackrel{\eqref{eq:tilde_m>}}{\ge}&
\alpha(d+e)+\tilde{m}(e,\iota,d)\\
&=&\alpha(d+e)+\tilde{m}(e,\iota,\delta+\beta e)\\
& \stackrel{\eqref{eq:tilde_m>}}{\ge}&
\alpha((\beta e +\delta)+e)+\big(\beta(\iota+e)+\tilde{m}(e,\iota,\delta)\big)\\
& \stackrel{\eqref{eq:ide2}}{\ge}&
\alpha((\beta e +\delta)+e)+\beta(\iota+e)+\left(\iota+\delta+\left\lfloor \frac{\iota\delta}{e} \right\rfloor\right)\\
&=& (\alpha e + \iota)+(\beta e + \delta)+  \left\lfloor \alpha\beta e+ \alpha\delta+\beta\iota + \frac{\iota\delta}{e} \right\rfloor\\
&=& i+d+ \left\lfloor \frac{id}{e} \right\rfloor,
\end{eqnarray*}
and \eqref{eq:goal} is verified.
\end{proof}


\subsection{Proof of Theorem~\ref{thm:main}}

\begin{lem}\label{lem:parity}
If $r \in \mathbb{N} \setminus \{0\}$ is odd and $s \in \mathbb{N}\setminus \{0\}$ is even, then $\displaystyle\left\lfloor \frac{r-1}{s} \right\rfloor=\left\lfloor \frac{r}{s} \right\rfloor$.
\end{lem}

\begin{proof}
Assume that $\left\lfloor \frac{r-1}{s} \right\rfloor< \left\lfloor \frac{r}{s} \right\rfloor$. Then the number $t:=\left\lfloor \frac{r}{s} \right\rfloor \in \mathbb{N}$ satisfies $\frac{r-1}{s} < t \le \frac{r}{s}$, whence $r-1 < st \le r$ and in turn $r=st$. But this is incompatible with the parities of $r$ and $s$.
\end{proof}

\begin{proof}[Proof of Theorem~\ref{thm:main}(a)]
Clearly, $m(n,p)= \sup\{m(n,i,d): i,d \in \mathbb{N}, i+d=p\}$. Proposition~\ref{prop} gives
\begin{align*}
m(n,p)&=\sup\left\{p+\left\lfloor \frac{id}{n-p} \right\rfloor: i,d \in \mathbb{N}, i+d=p\right\}\nonumber\\
&=\left\{
\begin{array}{l@{\;}ll}
\displaystyle p+\left\lfloor \frac{\frac{p}{2}\cdot \frac{p}{2}}{n-p} \right\rfloor &\displaystyle = p+\left\lfloor \frac{p^2}{4(n-p)} \right\rfloor& \text{if } p \text{ is even},\\[2ex]
\displaystyle p+\left\lfloor \frac{\frac{p-1}{2}\cdot \frac{p+1}{2}}{n-p} \right\rfloor &\displaystyle =p+\left\lfloor \frac{p^2-1}{4(n-p)} \right\rfloor & \text{if } p \text{ is odd}.
\end{array}
\right.
\end{align*}
Now Lemma~\ref{lem:parity} completes the proof.
\end{proof}

\begin{proof}[Proof of Theorem~\ref{thm:main}(b)]
Clearly, $p(m,n) \le \min\{m,n\}$, since $R_{m \times n}$ can be covered by $\min\{m,n\}$ stripes of width $1$ and length $\max\{m,n\}$. So,
\begin{equation}\label{eq:pf1b_1}
p(m,n)=\min\big\{p \in \{0,1,\ldots,\min\{m,n\}\}: m(n,p) \ge m\big\}.
\end{equation}

We use Theorem~\ref{thm:main}(a). Note that, for all $m,n,p \in \mathbb{N}$ with $p \le \min\{m,n\}$,
$$
m \le m(n,p) \;\;\Leftrightarrow\;\;  m \le p+\left\lfloor \frac{p^2}{4(n-p)} \right\rfloor
\;\;\Leftrightarrow\;\;  m \le p+\frac{p^2}{4(n-p)}
\;\;\Leftrightarrow\;\;  p^2-\frac{4}{3}(m+n)p+\frac{4}{3}mn \le 0.
$$
The last applies if and only if $p$ is between the zeros of the quadratic polynomial; that is,
$$
\frac{2}{3}\left(m+n-\sqrt{m^2+n^2-mn}\right) \le p \le \frac{2}{3}\left(m+n+\sqrt{m^2+n^2-mn}\right).
$$
The second inequality is always satisfied, because $p \le \min\{m,n\}$. Therefore,
$$
m \le m(n,p)
\quad\Leftrightarrow\quad p \ge \left\lceil \frac{2}{3}\left(m+n-\sqrt{m^2+n^2-mn}\right) \right\rceil.
$$
Combining this with \eqref{eq:pf1b_1} yields the claim
\end{proof}




\begin{thebibliography}{00}

\bibitem{ambrus+}
G.\ Ambrus, I.\ B\'ar\'any, P.\ Frankl, D.\ Varga: \emph{Piercing the chessboard.} arXiv:2111.097002v1.

\bibitem{barany_frankl}
I.\ B\'ar\'any, P.\ Frankl: \emph{How (not) to cut your cheese.}  Amer.\ Math.\ Monthly {\bf 128} (2021), no.\ 6, 543--552. 

\bibitem{bertram+}
A.\ Bertram, I.\ Ciocan-Fontanine, W.\ Fulton:
\emph{Quantum multiplication of Schur polynomials.}
J.\ Algebra {\bf 219} (1999), no.\ 2, 728--746.

\bibitem{golomb} 
S.W.\ Golomb: \emph{Polyominoes. Puzzles, Patterns, Problems, and Packings.} With diagrams by Warren Lushbaugh. Second edition. With an appendix by Andy Liu. Princeton University Press, Princeton, NJ, 1994.
  
\bibitem{gruenbaum_shephard}
B.\ Gr\"unbaum, G.C.\ Shephard: \emph{Tilings and patterns.}  W.H.\ Freeman and Company, New York, 1987.

\bibitem{james+}
G.\ James, A.\ Kerber: \emph{The Representation Theory of the Symmetric Group.} Addison-Wesley, Reading, MA, 1981.
\bibitem{lascoux+}
A. Lascoux, B.\ Leclerc, J.-Y.\ Thibon:
\emph{Ribbon tableaux, Hall-Littlewood functions, quantum affine algebras, and unipotent varieties.}
J.\ Math.\ Phys.\ {\bf 38} (1997), no.\ 2, 1041--1068.

\bibitem{martin}	
G.E.\ Martin: \emph{Polyominoes. A guide to puzzles and problems in tiling.} MAA Spectrum. Mathematical Association of America, Washington, DC, 1991.
	
\bibitem{pak}
I.\ Pak: \emph{Ribbon tile invariants.}
Trans.\ Amer.\ Math.\ Soc.\ {\bf 352} (2000), no.\ 12, 5525--5561.

\bibitem{sheffield}
S. Sheffield: \emph{Ribbon tilings and multidimensional height functions.}
Trans.\ Amer.\ Math.\ Soc.\ 354 (2002), no.\ 12, 4789--4813. 

\bibitem{stanton_white}	
D.W.\ Stanton, D.E.\ White: \emph{A Schensted algorithm for rim hook tableaux.}
J.\ Combin.\ Theory Ser.\ A {\bf 40} (1985), no.\ 2, 211--247. 
	
\end{thebibliography}
\end{document}